\documentclass{amsart}

\usepackage{graphicx}
\usepackage{hyperref}
\usepackage{subfigure}
\usepackage{amsmath}
\usepackage{mathrsfs} 
\usepackage{amsfonts}
\usepackage{amssymb}%
\usepackage{amsthm}
\usepackage{amscd}
\usepackage{verbatim}
\usepackage{url}
 \usepackage{mathabx}
 \usepackage{dsfont}
 \usepackage[all]{xypic} 
\usepackage{endnotes}
\usepackage{enumerate}
\usepackage{mathtools}
\usepackage{comment}
\usepackage[latin1]{inputenc}
\usepackage{tikz}
\usetikzlibrary{shapes,arrows}
\usepackage{color}
\usepackage{alltt}
\usepackage{algorithm}
\usepackage[noend]{algpseudocode}

\newcommand{\C}{\mathbb{C}}

\renewcommand{\P}{{\sf{P}}}
\newcommand{\ZZ}{\mathbb{Z}}

\newcommand{\ot}{\otimes}

\newcommand{\ra}{\rightarrow}

\newcommand{\sDet}{\operatorname{sDet}}

\newcommand{\op}{\oplus}
\newcommand{\bop}{\bigoplus}
\newcommand{\Tr}{\textnormal{Tr}}

\newcommand{\bra}[1]{\mbox{$\langle #1|$}}
\newcommand{\ket}[1]{\mbox{$|#1\rangle$}}

\theoremstyle{plain}
\newtheorem{theorem}{Theorem}[section]
\newtheorem{corollary}[theorem]{Corollary}
\newtheorem{proposition}[theorem]{Proposition}

\theoremstyle{definition}
\newtheorem{defn}[theorem]{Definition}
\theoremstyle{definition}

\theoremstyle{definition}
\newtheorem{example}[theorem]{Example}
\theoremstyle{definition}

\title[Tutte Polynomial With Determinantal Circuits]{Computing the Tutte Polynomial of Lattice Path Matroids Using Determinantal Circuits}
\author[Jason Morton and Jacob Turner]{Jason Morton and Jacob Turner}

\date{March 1, 2015}
\begin{document}
\maketitle 
\begin{center}
\vskip -2em
 \small{Department of Mathematics, Pennsylvania State University, University Park PA 16802}
\end{center}

\begin{abstract}
We give a quantum-inspired $O(n^4)$ algorithm computing the Tutte polynomial of a lattice path matroid, where $n$ is the size of the ground set of the matroid.   Furthermore, this can be improved to $O(n^2)$ arithmetic operations if we evaluate the Tutte polynomial on a given input, fixing the values of the variables.  The best existing algorithm, found in 2004, was $O(n^5)$, and the problem has only been known to be polynomial time since 2003.  Conceptually, our algorithm embeds the computation in a determinant using a recently demonstrated equivalence of categories useful for counting problems such as those that appear in simulating quantum systems.
\end{abstract}

\begin{center}
\end{center}


\section{Introduction}

Since their introduction in the early $20^{th}$ century, matroids have proven to be immensely useful objects generalizing notions of linear independence. They have become ubiquitous, appearing in fields from computer science and combinatorics to geometry and topology (cf. \cite{brylawski1992tutte,oxley1992matroid}).

Perhaps the most famous invariant of a matroid $M$ is called the Tutte polynomial, $T_M(x,y)$. The polynomial was originally defined as an invariant of graphs, generalizing the chromatic polynomial \cite{tutte1954contribution}. This was later discovered to specialize to the Jones polynomial of an associated alternating knot (\cite{thistlethwaite1987spanning}) as well as the partition function of the Potts model in statistical physics, the random cluster model in statistical mechanics (\cite{fortuin1972random}), the reliability polynomial in network theory (\cite{colbourn1987combinatorics}), and flow polynomials in combinatorics (\cite{welsh1993complexity}). In fact, it is the most general invariant of matroids such that $F(M\op M')=F(M)F(M')$; all other such invariants are evaluations of the Tutte polynomial \cite{white1987combinatorial}. 

The Tutte polynomial specializes to a generating function of special configurations of chip firing games on graphs \cite{lopez1997chip}. In convex geometry, it also relates to the Ehrhart polynomial of zonotopes which is used for calculating integral points in polytopes \cite{moci2012tutte}. More recently the connection with quantum simulation and computation has been explored, but without obtaining a new classical algorithm for the Tutte polynomial of lattice path matroids \cite{sokal2005multivariate,wocjan2006jones,aharonov2007polynomial,aharonov2009polynomial}.

In addition to generalizing several polynomials, specific values of the Tutte polynomial give information about graphs, including the number of (spanning) forests, spanning subgraphs, and acyclic orientations. However, computing the Tutte polynomial for general matroids is very difficult. For $x,y$ positive integers, calculating $T_M(x,y)$ is $\#\P$-hard \cite{jaeger1990computational}. As such, a large amount of work has been done to determine when Tutte polynomials can be efficiently computed.

Lattice path matroids were presented in \cite{bonin2006lattice,bonin2003lattice} as a particularly well-behaved and yet very interesting class of matroids. For lattice path matroids, the computation of the Tutte polynomial was shown to be polynomial time in the 2003 paper \cite{bonin2003lattice}. In \cite{bonin2007multi}, it was proven that the time complexity of computing the Tutte polynomial is $O(n^5)$, where $n$ is the size of the ground set of the matroid. 

We give a quantum-inspired $O(n^4)$ algorithm computing the Tutte polynomial of a lattice path matroid, where $n$ is the size of the ground set of the matroid.   Furthermore, this can be improved to $O(n^2)$ arithmetic operations (as opposed to bit operations) if we evaluate the Tutte polynomial on a given input, fixing the values of the variables.  The best existing algorithm was $O(n^5)$.  Our algorithm embeds the computation in a determinant using a recently demonstrated equivalence of categories  \cite{morton2013det} useful for counting problems such as those addressed by holographic algorithms.  

\section{Background}

Leslie Valiant defined matchgates to simulate certain quantum circuits efficiently \cite{ValiantFOCS2004}. He then defined holograph algorithms as a way of finding polynomial-time algorithms for certain $\#\P$ problems \cite{MR2120307}. He later went on to demonstrate several new polynomial-time algorithms for problems for which no such algorithms were previously known \cite{valiant2006accidental}. These algorithms came to be known as matchcircuits and have been studied extensively \cite{MR2277247,MR1932906,Bravyi2008,JozsaMiyake}. They were reformulated later in terms of tensor networks and are equivalent to Pfaffian circuits \cite{landsberg2012holographic,morton2010pfaffian}.

Tensor networks were probably first introduced by Roger Penrose in the context of quantum physics \cite{penrose1971applications}, but appear informally as early as Cayley.  They have seen many applications in physics,  as they can represent channels, maps, states and processes appearing in quantum theory~\cite{2005PhRvA..71d2337G, 2006PhRvA..73e2309G, 2007PhRvL..98v0503G,biamonte2011categorical}.

Tensor networks also generalize notions of circuits, including quantum circuits. As such, they have seen an increase in popularity as tools for studying complexity theory \cite{bulatov2005complexity,dyer2010effective,cai2011non}. While we do not define tensor networks in this paper, the material should be understandable without a detailed knowledge of this formalism. For more detailed expositions on tensor networks, see \cite{selinger2009survey,joyal1991geometry,joyalgeometry,biamonte2012invariant}.

\subsection{Structure} In \cite{morton2013det}, a new type of circuit was defined based on determinants, as opposed to matchcircuits which are defined in terms of Pfaffians. It was shown that these circuits had polynomial time evaluations.  In this paper, we give an algorithm that improves the complexity of computing the Tutte polynomial of a lattice path matroid to $O(n^{4})$ using these \emph{determinantal circuits}. Then we show that evaluating the Tutte polynomial on a specific input (fixed values of $x$ and $y$) can be done in $O(n^2)$ arithmetic operations. The paper is organized as follows: first we discuss weighted lattice paths and their relation to determinantal circuits. Then we recall the definitions of lattice path matroids and the relevant theorems from \cite{bonin2003lattice,bonin2007multi}. Lastly, we give an explicit algorithm for computing the Tutte polynomial and analyze its time complexity.

\section{Weighted Lattice Paths}

\tikzstyle{block} = [rectangle, 
    text width=1em, text centered, minimum height=1em]
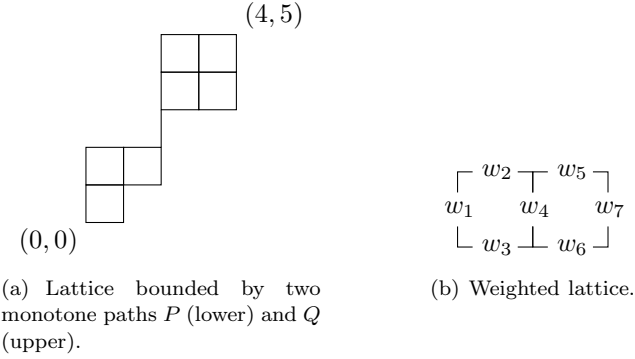
\begin{figure}
 \centering
 \subfigure[Lattice bounded by two monotone paths $P$ (lower) and $Q$ (upper).]{
\begin{tikzpicture}
\draw (-.5,-.25) node{$(0,0)$};
 \draw (0,0) rectangle (.5,.5);
 \draw (0,.5) rectangle (.5,1);
 \draw (.5,.5) rectangle (1,1);
 \draw (1,1) -- (1,1.5);
 \draw (1,1.5) rectangle (1.5,2);
 \draw (1,2) rectangle (1.5,2.5);
 \draw (1.5,1.5) rectangle (2,2);
 \draw (1.5,2) rectangle (2,2.5);
 \draw (2.5,2.75) node{$(4,5)$};
\end{tikzpicture}
}\qquad\qquad
\subfigure[Weighted lattice.]{
 \begin{tikzpicture}
 \node [block] (w1) {$w_1$};
 \node [block, right of =w1] (w4) {$w_4$};
 \node [block, right of =w4] (w7) {$w_7$};
 \node [block] (w2) at (.5,.5) {$w_2$};
 \node [block, below of=w2] (w3) {$w_3$};
 \node [block, right of=w2] (w5) {$w_5$};
 \node [block, right of=w3] (w6) {$w_6$};
 \draw (w1) |- (w2);
 \draw (w1) |- (w3);
 \draw (w2) -| (w4);
 \draw (w3) -| (w4);
 \draw (w4) |- (w5);
 \draw (w4) |- (w6);
 \draw (w5) -| (w7);
 \draw (w6) -| (w7);
 
 \end{tikzpicture}
}
\caption{Lattices}\label{fig:lattice}
\end{figure}

Let us consider $\ZZ^2$ as an infinite graph where two points are connected if they differ by $(\pm1,0)$ or $(0,\pm 1)$. Suppose we are given two monotone paths on $\ZZ^2$, $P$ and $Q$, that both start at $(0,0)$ and end at $(m,r)$. Furthermore, suppose that $P$ is never above $Q$ in the sense that there are no points $(p_1,p_2)\in P$, $(q_1,q_2)\in Q$ such that $p_1-q_1<0$ and $p_2-q_2>0$. We are interested in subgraphs of $\ZZ^2$ bounded by such pairs of paths. From here on out, ``lattice'' means of subgraph of this form. An example is given by Figure \ref{fig:lattice}(a).

Let $E$ be the set of edges of a lattice $G$. Suppose for each $e\in E$, we assign it a weight, $w(e)$. We call this a \emph{weighted lattice}. Given a monotone path $C\subseteq G$, we define the weight of $C$ to be the product of its edge weights $$w(C)=\prod_{e\in C}{w(e)}.$$

\begin{defn}
 Let $G$ be a lattice bounded by two paths with common endpoint $(m,r)$. A \emph{full path} in $G$ is a monotone path from $(0,0)$ to $(m,r)$.
\end{defn}

\begin{defn}
 Let $\mathscr{F}$ be the set of full paths of a weighted lattice $G$. The \emph{value} of $G$ is defined to be $$\sum_{C\in\mathscr{F}}{w(C)}.$$
\end{defn}

In Figure \ref{fig:lattice}(b), there are three full paths of the weighted lattice. The value of this lattice is $w_1w_2w_5+w_3w_4w_5+w_3w_6w_7$. 

We describe a way to assign matrices to every vertex of a weighted lattice in order to encode the weights of each edge. Let $G=(V,E)$ be a weighted lattice. For $v\in V$, we define the \emph{incoming edges} of $v$ to be those edges below or to the left of $v$ incident to $v$. The other edges incident to $v$ are the \emph{outgoing edges} of $v$.

The matrix we associate to $v$ has rows equal to the number of incoming edges and columns equal to the number of outgoing edges. We order the incoming edges of $v$ counterclockwise starting with the incoming edge closest to the negative $x$-axis. We order the outgoing edges of $v$ clockwise starting with the outgoing edge closest to the positive $y$-axis. This order defines how to associate the edges of $v$ with rows and columns of the matrix. We fill each column with the weight of the outgoing edge of $v$ it corresponds to, see Figure \ref{fig:matrixassign}

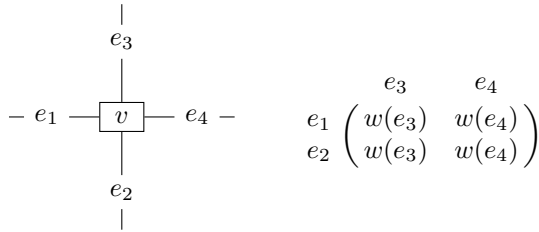
\begin{figure}
\centering
\begin{tikzpicture}
 \node [block] (w1) {$e_1$};
 \node [block,draw, right of=w1] (v) {$v$};
 \node [block, above of=v] (w3) {$e_3$};
 \node [block, below of=v] (w2) {$e_2$};
 \node [block, right of=v] (w4) {$e_4$};
 \draw (w1) -- (v);
 \draw (w2) -- (v);
 \draw (w3) -- (v);
 \draw (w4) -- (v);
 \draw (-.5,0) -- (w1);
 \draw (w4) -- (2.5,0);
 \draw (w3) -- (1,1.5);
 \draw (w2) -- (1,-1.5);
 \node at (5,0) {$\bordermatrix{ &e_3&e_4\cr
e_1&w(e_3)&w(e_4)\cr
e_2&w(e_3)&w(e_4)}$};

\end{tikzpicture}
\caption{A vertex and its associated matrix}\label{fig:matrixassign}
\end{figure}

Note that every edge in a lattice is the outgoing edge of precisely one vertex, so given the matrices associated to the vertices, the weight on the edges can be recovered. For a vertex $v$, we denote the matrix associated with it by $M_v$.

However, for the matrices associated with $(0,0)$ and $(m,r)$, we do something slightly different. As defined, $M_{(0,0)}$ would have zero rows and $M_{(m,r)}$	zero columns. We define $M_{(0,0)}$ to have one row with the weights of the outgoing edges in the appropriate columns. We define $M_{(m,r)}$ to have one column, with all entries 1.

\subsection{Determinantal Circuits} We can turn a weighted lattice into a determinantal circuit \cite{morton2013det} in such a way that the value of the determinantal circuit is the value of the weighted lattice.

A determinantal circuit can be given as a series of matrices $\{S_i\}_{i=1}^s$, which are called \emph{stacks}. We define the function $\sDet(M):\operatorname{Mat}(n,m) \ra \operatorname{Mat}(2^n,2^m)$ by, if $M$ is an $n\times m$ matrix, 
$$\sDet(M)=\sum_{I\subseteq[n],J\subseteq[m]}{M_{I,J}|I\rangle\langle J|}$$ where $M_{I,J}$ is the minor of $M$ including rows $I$ and columns $J$ and where $\ket{I}=\bigotimes_{i \in [n]}{ u_{i,\chi(i,I)}}$, and $\bra{J}=\bigotimes_{i \in [m]}{ u^*_{i,\chi(i,J)}}$. Here $u_{i,0},u_{i,1}$ is an orthonormal basis for the $i$th copy of $\C^2$ in the tensor product $(C^2)^{\ot n}$ and $u_{i,0}^*,u_{i,1}^*$ is an orthonormal basis for the $i$th copy of $(\C^2)^*$ in the tensor product $((C^2)^*)^{\ot m}$. 
The indicator function $\chi(i,I)=1$ if $i \in I$ and $0$ otherwise. 

The function $\sDet$ is applied to each stack and is functorial.  That is, $\sDet$ defines a monoidal functor between a category whose morphisms are $n\times m$ matrices (with monoidal product the direct sum) to a category whose morphisms are $2^n\times 2^m$ matrices (with monoidal product the Kronecker product) that respects matrix multiplication, see \cite{morton2013det} for details. The value of the determinantal circuit is the trace of the composition of the $\sDet(S_i)$: 
\[
\Tr\bigg(\prod_{i=1}^s{ \sDet S_i}\bigg). 
\]
While the $\sDet$ function takes a matrix to an exponentially larger matrix, the following proposition still allows us to compute the value of determinantal circuit efficiently.
\begin{proposition}[\cite{morton2013det}]
\[
\Tr\bigg(\prod_{i=1}^s{ \sDet S_i}\bigg) = 
\det\bigg(I+\prod_{i=1}^s{S_i}\bigg).
\]
\end{proposition}

Take a weighted lattice with the matrices described above assigned to each vertex. Figure 2 shows an example of how to determine the stacks of the determinantal circuit from the lattice. For each of the diagonal arrows, take the matrix direct sum of the matrices $M_v$ along the direction of the arrow, i.e., $M_v$ will be block diagonal. If $G$ is a weighted lattice, let $D_G$ be its associated determinantal circuit.

Let $G$ be a weighted lattice bounded by two paths $P$ and $Q$ from $(0,0)$ to $(m,r)$, and $\mathscr{F}$ the set of full paths of $G$. We also add an edge connecting $(0,0)$ and $(m,r)$ to get a picture as in Figure \ref{fig:detcirc}. We denote this edge by 1, and we can view the set of full paths as cycles where 1 is the final edge. If $C\in\mathscr{F}$, we describe it as a series of triples $(v_i,e_i,e_{i+1})$ where $e_i$ is the $i$th edge and $v_i$ is the common vertex of $e_i$ and $e_{i+1}$, $i$ ranging from 1 to $n=m+r$. We also include the triples $(v_1,1,e_1)$ and $(v_n,e_n,1)$. The following is a special case of Proposition 4.3 in \cite{morton2013det}.

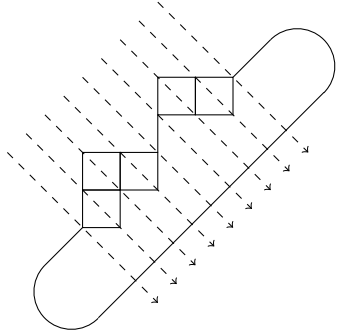
\begin{figure}
 \centering

\begin{tikzpicture}

 \draw (0,0) rectangle (.5,.5);
 \draw (0,.5) rectangle (.5,1);
 \draw (.5,.5) rectangle (1,1);
 \draw (1,1) -- (1,1.5);
 \draw (1,1.5) rectangle (1.5,2);
 \draw (1.5,1.5) rectangle (2,2);
\draw[dashed,->] (-1,1) -- (1,-1);
\draw[dashed,->] (-.75,1.25) -- (1.25,-.75);
\draw[dashed,->] (-.5,1.5) -- (1.5,-.5);
\draw[dashed,->] (-.25,1.75) -- (1.75,-.25);
\draw[dashed,->] (0,2) -- (2,0);
\draw[dashed,->] (.25,2.25) -- (2.25,.25);
\draw[dashed,->] (.5,2.5) -- (2.5,.5);
\draw[dashed,->] (.75,2.75) -- (2.75,.75);
\draw[dashed,->] (1,3) -- (3,1);
\draw (-.5,-.5) -- (0,0);
\draw (-.5,-.5) arc (135:315:.5);
\draw (2.5,2.5) -- (2,2);
\draw (2.5,2.5) arc (135:-45:.5);
\draw (.21,-1.2) -- (3.21,1.8);
\end{tikzpicture}

\caption{A Determinantal Circuit}\label{fig:detcirc}
\end{figure}

\begin{proposition}\label{prop:weightedlattice}
The value of $D_G$ is given by the expression $$\sum_{C\in \mathscr{F}}{\bigg(\det(M_{v_1,1,e_1})\det(M_{v_{n},e_{n},1})\prod_{(v_i,e_i,e_{i+1})\in P}{\det(M_{v_i,e_i,e_{i+1}})}\bigg)}$$ where $M_{v_i,e_i,e_{i+1}}$ is the minor of $M_{v_i}$ specified by the edges $e,e'$.
\end{proposition}
\begin{corollary}
 The value of $D_G-1$ is equal to the value of $G$.
\end{corollary}
\begin{proof}
 By the way we constructed $M_{v_i}$, $M_{v_i,e_i,e_{i+1}}$ is a $1\times 1$ minor with entry $w(e_{i+1})$. Furthermore $M_{v_1,1,e_1}$ has single entry $w(e_1)$ and $M_{v_n,e_n,1}$ has single entry 1. Thus 
 $$\sum_{C\in \mathscr{F}}{\bigg(\det(M_{v_1,1,e_1})\det(M_{v_{n},e_{n},1})\prod_{(v_i,e_i,e_{i+1})\in P}{\det(M_{v_i,e_i,e_{i+1}})}\bigg)}$$
 $$=\sum_{C\in\mathscr{F}}{\prod_{i=1}^n{w(e_i)}}=\sum_{C\in\mathscr{F}}{w(C)}.$$ However, the case of the empty path is counted in this value, so if we subtract one from the value of $D_G$, we get the value of $G$.
\end{proof}

To illustrate Proposition \ref{prop:weightedlattice}, we consider the following example.
\pagebreak

\begin{example}
 Suppose we have the following determinantal circuit, $D$:
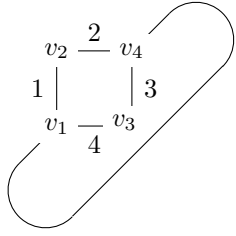
\begin{center}
\begin{tikzpicture}

 \draw (0,0) rectangle (1,1);
\draw (-.5,-.5) -- (0,0);
\draw (-.5,-.5) arc (135:315:.5);
\draw (1.5,1.5) -- (1,1);
\draw (1.5,1.5) arc (135:-45:.5);
\draw (.21,-1.2) -- (2.21,.8);
\draw (-.25,.5) node{1};
\draw (.5,1.25) node{2};
\draw (1.25,.5) node{3};
\draw (.5,-.25) node{4};
\node[rectangle, fill=white] (v1) at (0,0) {$v_1$};
\node[rectangle, fill=white] (v4) at (1,1) {$v_4$};
\node[rectangle, fill=white] (v2) at (0,1) {$v_2$};
\node[rectangle, fill=white] (v3) at (.9,.05) {$v_3$};
\end{tikzpicture}
\end{center}

There are two full paths: $v_1\to v_2\to v_4$ and $v_1\to v_3\to v_4$ with weights $1\cdot 2=2$ and $3\cdot 4=12$, respectively. Therefore, the value of this weighted lattice is $2+12=14$. Multiplying the matrices associated to the vertices in this determinantal circuit, we get 
\begin{equation*}
 \begin{pmatrix}
  1&4
 \end{pmatrix}
\begin{pmatrix}
 2&0\\
 0&3
\end{pmatrix}
\begin{pmatrix}
 1&1
\end{pmatrix}=
\begin{pmatrix}
 14	
\end{pmatrix}
\end{equation*}
noting that $M_{1,0}\oplus M_{0,1}=\begin{pmatrix}2&0\\0&3\end{pmatrix}$. Then $D_G=14+1$ and thus $D_G-1$ is clearly the value of the weighted lattice.

\end{example}

\section{Lattice Path Matroids}

Recall that a {\em matroid} 
may be defined as a pair $(G,\mathcal{B})$ where $G$, the ground set, is a finite set and $\mathcal{B}$, the bases, are a collection of subsets of $G$ such that (i) $\mathcal{B}$ is nonempty, (ii) if $A,B\in\mathcal{B}$ and there is an element $x\in A\setminus B$, then there exists an element $y\in B\setminus A$ such that $A-x+y\in\mathcal{B}$. 

Lattice path matroids are defined with respect to a lattice bounded by two monotone paths $P$ and $Q$ from $(0,0)$ to $(m,r)$ as described before. Given a full path in the region bounded by $P$ and $Q$, it can be described as a subset $B\subseteq [m+r]$ as follows: The path associated to $B$ is the sequence of steps $s_1s_2\cdots s_{m+r}$ where $s_i$ is a north step if $s_i\in B$ and an east step otherwise.
\begin{defn}[\cite{bonin2003lattice}]
Let $P=p_1\cdots p_n$ and $Q=q_1\cdots q_n$ be two lattice paths from $(0,0)$ to $(m,r)$ with $P$ never going above $Q$. Define $M[P,Q]$ to be the transversal matroid with ground set $[m+r]$ and bases those $B\subseteq[m+r]$ that represent full paths in the region bounded by $P$ and $Q$. A \emph{lattice path matroid} is any matroid isomorphic to some $M[P,Q]$.
\end{defn}

Lattice path matroids are a very nice example of matroids and of interest is the Tutte polynomial of such matroids. It turns out that this can be given as the value of a particular weighting of the lattice defining the matroid. The following is a slight restatement of the original theorem:

\begin{theorem}[\cite{bonin2003lattice}]\label{thm:latticetutte}
 The Tutte polynomial of a lattice path matroid $M[P,Q]$ is the value of the lattice $G$ defined by $P$ and $Q$ with the north steps of $Q$ having weight $x$, the east steps of $P$ having weight $y$, and all other lattice weights equal to 1.
\end{theorem}

An example is shown in Figure \ref{fig:tuttelattice}. Beside the edges in bold are the corresponding weights. The weight of each non-bold edge is simply 1.
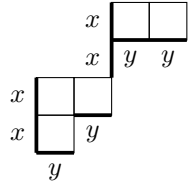
\begin{figure}
 \centering
 \begin{tikzpicture}
 \draw (0,0) rectangle (.5,.5);
 \draw (0,.5) rectangle (.5,1);
 \draw (.5,.5) rectangle (1,1);
 \draw (1,1) -- (1,1.5);
 \draw (1,1.5) rectangle (1.5,2);
 \draw (1.5,1.5) rectangle (2,2);
 \draw (-.25,.25) node{$x$};
 \draw (-.25,.75) node{$x$};
 \draw (.75,1.25) node{$x$};
 \draw (.75,1.75) node{$x$};
 \draw (.25,-.25) node{$y$};
 \draw (.75,.25) node{$y$};
 \draw (1.25,1.25) node{$y$};
 \draw (1.75,1.25) node{$y$};
 \path[draw=black,solid,line width=.5mm,fill=black] (0,0) -- (0,1);
\path[draw=black,solid,line width=.5mm,fill=black](1,1) -- (1,2);
 \path[draw=black,solid,line width=.5mm,fill=black] (0,0) -- (.5,0);
\path[draw=black,solid,line width=.5mm,fill=black] (.5,.5) -- (1,.5);
 \path[draw=black,solid,line width=.5mm,fill=black] (1,1.5) -- (2,1.5);
 \end{tikzpicture}
\caption{Weighted lattice with Tutte polynomial $(x^2 + xy + y^2 + x + y)(x)(x + y + y^2)$ as its value.}\label{fig:tuttelattice}
\end{figure}

\section{Algorithm for Computing the Tutte Polynomial}

We now give the algorithm for computing the Tutte polynomial in pseudocode. We assume that we are given as input two monotone paths $P$ and $Q$ from $(0,0)$ to $(m,r)$ such that $P$ never goes above $Q$. Both $P$ and $Q$ are given as length two lists. For $i\in[m+r]$, the entry $P[i,1]$ is the number of east steps and $P[i,2]$ is the number of north steps in path $P$ after $i$ steps. Similarly for $Q[i,1]$ and $Q[i,2]$. 

\begin{algorithm}[t] 
\caption{
 Compute the Tutte polynomial of a lattice path matroid defined by monotone paths $P$ and $Q$ from $(0,0)$ to $(m,r)$ such that $P$ never goes above $Q$.}\label{thealgo}

\begin{algorithmic}
 \State{$T \gets$ length-one row vector $(1)$}
 \For{$i$ \textbf{from} $0$ \textbf{to} $m+r$}
 \Comment{loop over all stacks}
 \State{$A\gets$ empty list}
 \For{$j$ \textbf{from} $0$ \textbf{to} $Q[i,2]-P[i,2]$}
 \Comment{loop over all vertices in stack $i$}
 \State{$v_j\gets(Q[i,1]+j,Q[i,2]-j)$}
 \Comment{the $j$th vertex in the $i$th stack}
 \State{append $M_{v_j}$ to $A$}
 \EndFor
 \State{$T\gets (T_1 A[1],\dotsc, T_\ell A[\ell])$}
 \Comment{$T_i$s are the correct partition of $T$ (see proof)}
 \EndFor
 \State{\Return{first entry of $T$}}
\end{algorithmic}
\end{algorithm}

\begin{theorem}\label{maintheorem}
Algorithm \ref{thealgo} computes the Tutte polynomial of a lattice path matroid in $O(n^4)$ time and evaluates it at fixed values of $x$ and $y$ in $O(n^2)$ arithmetic operations, where $n$ is the size of the ground set of the matroid.
\end{theorem}

\begin{proof}

Let $n=m+r$ be the length of $P$ and $Q$. The $i$th stack is always the vertices $i$ steps away from $(0,0)$, so there are $m+r$ stacks and $i$ iterates over each stack moving left to right. The number of vertices in the $i$th stack is the how many more north steps $Q$ has made than $P$, i.e. $Q[i,2]-P[i,2]$. Then the inner \texttt{for} loop iterates over these vertices, from top to bottom. For each vertex in the stack, the matrix $M_{v_j}$ associated to the vertex is found by a constant-time lookup.

After the matrices are calculated, we could take their direct sum, giving a matrix $S_i$, and multiply it by $T$. For example, the matrices associated to the nine stacks in Figure \ref{fig:tuttelattice}, each of which is a direct sum of $M_v$, are multiplied to obtain 
\[
\begin{pmatrix}
x & y\cr
\end{pmatrix}
\begin{pmatrix}
x & 1 & 0\cr
0 & 0 & 1\cr
\end{pmatrix}
\begin{pmatrix}
1 & 0 & 0\cr
0 & 1 & y\cr
0 & 1 & y\cr
\end{pmatrix}
\begin{pmatrix}
1 & 0 \cr
1 & 0 \cr
0 & 1 \cr
\end{pmatrix}
\begin{pmatrix}
x \cr
x \cr
\end{pmatrix}
\begin{pmatrix}
x & y\cr
\end{pmatrix}
\begin{pmatrix}
1 & 0 & 0\cr
0 & 1 & y\cr
\end{pmatrix}
\begin{pmatrix}
1 & 0 \cr
1 & 0 \cr
0 & 1 \cr
\end{pmatrix}
\begin{pmatrix}
1 \cr
1 \cr
\end{pmatrix},
\]
which  equals the Tutte polynomial $(x^2 + xy + y^2 + x + y)(x)(x + y + y^2)$.

We can be more careful, however, to improve the running time since each matrix $S_i$ is block-diagonal with block size at most 2.  

 $T$ starts out as a (row) vector and each iteration of the main loop updates $T$ to a new vector. Since at each stage we are multiplying $T$ by a block diagonal matrix, we can partition $T$ into $T=(T_1,\dots, T_\ell)$ such that multiplying $T$ by $\bop{M_{v_i}}$ is the same as calculating $(T_1M_{v_1},\dots,T_\ell M_{v_\ell})$. Eventually, $T$ becomes a $1\times 1$ matrix and the algorithm then returns its sole entry. 

Algorithm \ref{thealgo}, has two nested for loops. The outer \texttt{for} loop  clearly runs in time $O(n)$.
For any stack, there are at most $n$ vertices in the stack, and for any $v_i$, $M_{v_i}$ has at most two rows and at most two columns. If we specify values for $x$ and $y$, calculating $(T_1M_{v_1},\dots,T_\ell M_{v_\ell})$ takes at most $16n$ computations. Thus the inner loop runs in time $O(n)$. Since the loops are nested, the overall time complexity is $O(n^{2})$ arithmetic operations.

Should we wish to compute the polynomial itself rather than its value, we need to account for multiplying intermediate polynomials by $1,x,$ or $y$, and adding them.  The polynomials in the intermediate vector $T$ after stack $i$ are of degree at most $i$ so have at most ${i+2 \choose i}$ terms.  The last application of an $S_i$ can involve adding quadratically large polynomials and so the additional cost is at most $O(n^2)$. 

Then updating the vector can be done in time $O(n^3)$ and this lies in a \texttt{for} loop that iterates at most $n$ times. Thus the overall algorithm is $O(n^4)$. 
\end{proof}

\section{Conclusion}

Lattice path matroids and generalizations have been the subject of research for some time due to their nice properties \cite{lawrence1984some,schweig2010h,stanley1977some}. For example, this class of matroids is connected and closed under minor, direct sums, and duals \cite{bonin2006lattice}.

In the introduction, we discussed the importance of not only computing the Tutte polynomial of a matroid, but the value of the Tutte polynomial evaluated at specific points. We addressed both of these problems and gave explicit algorithms for both problems, representing an improvement over the previously known algorithms.

We accomplished this by framing the problem in terms of tensor networks. Tensor networks give a natural way to formulate many counting problems of graphs, and combinatorial problems more generally. However, since computing the value of tensor network is $\#\P$-hard in general, there should be a lot of focus on finding classes of tensor networks with polynomial-time evaluations. These can be then be used to devise new polynomial-time algorithms for counting problems.

In this paper, we used a class of tensor networks inspired by Valiant's holographic algorithms. This paper gives a concrete application of our determinantal circuits to provide new and/or improved algorithms. These algorithms use only basic matrix operations, namely multiplication and determinants. This makes them attractive for algorithm design.

More generally, determinantal circuits are well suited to evaluating the values of weighted lattices. Any problem that can be efficiently encoded in such a computation is susceptible to an algorithm using determinantal circuits. Theorem \ref{thm:latticetutte} gives and explicit way to encode the Tutte polynomial of a lattice path matroid as the value of a particular weighted lattice. If other matroids can have their Tutte polynomials realized as the value of a weighted lattice, then a determinantal circuit can be used to compute it.

\noindent {\bf Acknowledgements.} This work was partially supported by the Air Force Office of Scientific Research under contract FA8650-13-M-1563.  We are grateful to two anonymous reviewers for helpful and detailed comments.

\bibliographystyle{custom} 
\bibliography{bibfile} 
\end{document}